\documentclass{article}

\usepackage{amsmath,amssymb,amsfonts,amsthm}
\usepackage{fullpage}
\usepackage{color}
\usepackage{tikz}
\usetikzlibrary{decorations.pathreplacing}
\usepackage{graphicx}

\renewcommand\vec[1]{\overrightarrow{#1}}
\newcommand\cev[1]{\overleftarrow{#1}}
\newcommand{\bk}{\bold{m}}
\newtheorem{theorem}{Theorem}[section]
\newtheorem{prop}[theorem]{Proposition}
\newtheorem{lemma}[theorem]{Lemma}

\theoremstyle{remark}
\newtheorem{remark}{Remark}
\newtheorem{example}{Example}
\newtheorem{definition}[theorem]{Definition}

\begin{document}

\title{Two dualities: Markov and Schur--Weyl}

\author{Jeffrey Kuan}

\date{}

\maketitle

\abstract{We show that quantum Schur--Weyl duality leads to Markov duality for a variety of asymmetric interacting particle systems. In particular, we consider three cases:

(1) Using a Schur--Weyl duality between a two--parameter quantum group and a two--parameter Hecke algebra from \cite{BW}, we recover the Markov self--duality of multi--species ASEP previously discovered in \cite{KIMRN} and \cite{BS3}.

(2) From a Schur--Weyl duality between a co--ideal subalgebra of a quantum group and a Hecke algebra of type B \cite{BWW}, we find a Markov duality for a multi--species open ASEP on the semi--infinite line. The duality functional has not previously appeared in the literature.

(3) A ``fused'' Hecke algebra from \cite{crampe2020fused} leads to a new process, which we call \textit{braided ASEP}.  In braided ASEP, up to $m$ particles may occupy a site and up to $m$ particles may jump at a time. The Schur--Weyl duality between this Hecke algebra and a quantum group lead to a Markov duality. The duality function had previously appeared as the duality function of the multi--species ASEP$(q,m/2)$ \cite{KIMRN} and the stochastic multi--species higher spin vertex model \cite{KuanCMP}.

}

\section{Statement of Results}

\subsection{Description of models}

\subsubsection{Multi--species ASEP}
We consider the multi--species ASEP with $n$ species of particles, labeled by $\{1,2,\ldots,n\}$. The particles occupy either a finite lattice with closed boundary conditions, or the infinite line $\mathbb{Z}$. At most one particle may occupy a site. The jump rates are described by (for $k>l$)
\begin{align*}
&k \ l \rightarrow l \ k \text{ with rate } q\\
&l \ k \rightarrow k \ l \text{ with rate } 1
\end{align*}
This means that particles with larger indices will drift left.

For any lattice site $x$, any integer $ j \in \{1,\ldots,n\}j$ and any particle configuration $\eta$, define
$$
\vec{N}_x^j (\eta)
$$
to be the number of particles of species $j$ in particle configuration $\eta$ that are located strictly to the right of $x$. Similarly, define
$$
\cev{N}_x^j(\eta)
$$
to be the number of particles of species $j$ in particle configuration $\eta$ that are located strictly to the \textit{left} of $x$.

For any particle configuration $\eta$, let $A_r(\eta)$ denote the set of lattice sites $x$ that contain a species $r$ particle.

\subsubsection{Multi--species open ASEP}
Again onsider an ASEP with multiple species of particles. This time, let the particles will be labeled by integers $\{-r,\ldots,-1,1\ldots,r\}$, while holes will be labeled by a ``0''. The particles occupy the lattice $\{0,1,2,\ldots L\}$ for some $L \in \mathbb{N}\cup\{\infty\}$ and at most one particle may occupy a site. In the bulk of the system, the jump rates are again described by (for $k>l$)
\begin{align*}
&k \ l \rightarrow l \ k \text{ with rate } q\\
&l \ k \rightarrow k \ l \text{ with rate } 1
\end{align*}
where $q \in (0,1)$. In words, this means that particles with a larger index will drift left. At the origin, the jump rates are described by (for $k>0$)
\begin{align*}
&k \rightarrow -k \text{ with rate } 1
&-k \rightarrow k \text{ with rate } Q
\end{align*}
where $Q\in (0,1)$. So particles with a positive index will generally leave the lattice over time and be replaced by their negatively indexed counterparts. %Denote this generator by $L_B^{\text{left}}$. 

Let
$$
\vec{N}_x^j (\eta),
\cev{N}_x^j(\eta),
A_k(\eta)
$$
have the same meanings as before. 

Note that this process previously appeared in \cite{bufetov2020interacting}.

\subsubsection{Braided ASEP}
The state space of braided ASEP consists of particle configurations on a one--dimensional lattice. Each lattice site may contain up to $m$ particles. To define the jump rates, first define some notation: the $q$--deformed integers, factorials, binomials and Pochhamer are defined by
$$
[n]_{q^2} =  \frac{1- q^{2n}}{1-q^2}, \quad \quad [n]_{q^2}^! = \{1\}_{q^2} \cdots \{n\}_{q^2}, \quad \quad \binom{n}{k}_{q^2} = \frac{[n]_{q^2}^!}{[k]_{q^2}^! [n-k]_{q^2}^!}.
\quad \quad (a;q)_n = (1-a)(1-aq)\cdots ( 1-aq^{n-1}).
$$
By convention, $[0]_q^!$ and $(a;q)_0$ both equal $1$, and $\binom{n}{k}_q=0$ if $k>n$. The total particle number is conserved, meaning that no particles are created or annihilated. The local jump rates are given by 

\begin{center}
\begin{tikzpicture}[scale=1, every text node part/.style={align=center}]
\usetikzlibrary{arrows}
\usetikzlibrary{shapes}
\usetikzlibrary{shapes.multipart}
\tikzstyle{arrow}=[->,>=stealth,thick,rounded corners=4pt]

\draw (0.33,0)--(1.66,0);
\draw (0.33,0)--(0.33,0.66);
\draw (1,0)--(1,0.66);
\draw (1.66,0)--(1.66,0.66);

\draw[fill=black] (0.66,0.33) circle (8pt);
\draw[fill=black] (0.66,1) circle (8pt);
\draw[fill=black] (0.66,1.66) circle (8pt);
\draw[fill=black] (1.33,0.33) circle (8pt);
\draw [decorate,decoration={brace,amplitude=10pt},xshift=-4pt,yshift=0pt]
(0,0) -- (0,2.0) node [black,midway,xshift=-0.6cm] 
{\footnotesize $k_1$};
\draw [decorate,decoration={brace,amplitude=10pt,mirror,raise=4pt},yshift=0pt]
(2,0) -- (2,2.0) node [black,midway,xshift=0.8cm] {\footnotesize
$k_2$};

\node at (4,1) {$\longrightarrow$};

\draw (6.33,0)--(7.66,0);
\draw (6.33,0)--(6.33,0.66);
\draw (7,0)--(7,.66);
\draw (7.66,0)--(7.66,0.66);

\draw[fill=black] (6.66,0.33) circle (8pt);
\draw[fill=black] (6.66,1) circle (8pt);
\draw[fill=black] (7.33,1) circle (8pt);
\draw[fill=black] (7.33,0.33) circle (8pt);
\draw [decorate,decoration={brace,amplitude=10pt},xshift=-4pt,yshift=0pt]
(6,0) -- (6,2.0) node [black,midway,xshift=-0.6cm] 
{\footnotesize $l_1$};
\draw [decorate,decoration={brace,amplitude=10pt,mirror,raise=4pt},yshift=0pt]
(8,0) -- (8,2.0) node [black,midway,xshift=0.8cm] {\footnotesize
$l_2$};

\end{tikzpicture}
\end{center}
at rate $$\binom{k_1}{l_2}_q (q^{2(m-k_2)};q^{-2})_{k_1-l_2} q^{2(m-k_2-k_1+l_2)l_2}  .$$ In this example, $k_1=3,k_2=1,l_1=l_2=2$, so this jump occurs at rate $(q^{-2} + 1 + q^2)(1 - q^{2(m-1)})q^{4(m-1)}q^2.$ Note that the jump rate is nonzero if and only if $l_2 \leq k_1$, and multiple particles may jump at a time. The particles will generally drift to the left over time. If $q=1$, the jump rate is only nonzero if $k_1=l_2$; so the only permissible jump is that all the particles at sites $x$ and $x+1$ switch places at rate $1$. This generalizes the usual SSEP, albeit in a somewhat trivial manner.

Let $\eta$ denote a particle configuration, where $\eta(x) \in \{0,1,\ldots,m\}$ denotes the number of particles at lattice site $x$. This should not be confused with the multi--species models, where $\eta(x)$ indicates the index of the particle (or hole) located at site $x$.

\subsection{Markov duality}

\begin{definition}
Two Markov processes $X(t)$ and $Y(t)$ on state spaces $\mathfrak{X}$ and $\mathfrak{Y}$ are \textit{Markov dual} with respect to a function $D(x,y)$ on $\mathfrak{X}\times \mathfrak{Y}$ if
$$
\mathbb{E}_x[D(X(t),y)] = \mathbb{E}_y[D(x,Y(t))]
$$ 
for all $t\geq 0$ and all $x\in \mathfrak{X}, y \in \mathfrak{Y}$. For discrete state spaces, there is an equivalent definition involving the generators. Let $L_X$ and $L_Y$ be the generators of $X(t)$ and $Y(t)$, viewed as $\mathfrak{X} \times \mathfrak{X}$ and $\mathfrak{Y} \times \mathfrak{Y}$ matrices, and view $D(x,y)$ as a $\mathfrak{X} \times \mathfrak{Y}$ matrix. Then $X(t)$ and $Y(t)$ are \textit{Markov dual} if 
$$
L_XD = DL_Y^T
$$
where the $^T$ denotes matrix transposition.
\end{definition}

\subsection{Main Theorems}

\subsubsection{Multi--species ASEP}
Define the function
\begin{equation}\label{MSch}
D(\eta,\xi) = \prod_{k=1}^n  \prod_{x \in A_k(\xi)} 1_{\eta(x) \geq \xi(x)} q^{ -2x - \sum_{j=k}^n 2\vec{N}_x^{j}(x) }
\end{equation}
\begin{theorem}\label{Sch}
The multi--species ASEP is self--dual with respect to the function in \eqref{MSch}.
\end{theorem}

This result is not new; it had been previously problem in \cite{KIMRN} and \cite{BS3}; see also \cite{KuanJPhysA} and \cite{BS,BS2} for the $n=2$ case. The single--species ($n=1$) case goes back to \cite{Sch97}. We state this result here because the proof is new, and it gives context for the duality result for the multi--species open ASEP. Note that up to multiplicative constants, 
$$
D(\eta,\xi) = \mathrm{const} \cdot \prod_{k=1}^n  \prod_{x \in A_k(\xi)} 1_{\eta(x) \geq \xi(x)} q^{ -2x + \sum_{j=k}^n 2\cev{N}_x^{j}(x) }
$$

\subsubsection{Multi--species open ASEP}
\begin{theorem}\label{OPEN}
For $r=1$, the open ASEP is self--dual with respect to the function 
\begin{multline*}
D(\eta,\xi) = 
1_{\{A_1(\xi) \subseteq A_1(\eta), A_{-1}(\xi)\subseteq A_{-1}(\eta)\}}  
 \prod_{x \in A_1(\xi)} (Q^{-2}q^{-2x-2\vec{N}_x^1(\eta)})  
\prod_{y \in A_{-1}(\xi)} q^{2y + 2\cev{N}_y^{-1}(\eta)} \\
\times q^{-2 \vert A_{-1}(\xi)\vert \cdot \vert A_1(\eta)\vert} 
\times q^{\vert A_1(\xi)\vert \cdot ( \vert A_1(\xi)\vert -1)} \prod_{y \in A_{-1}(\xi)} q^{2\vec{N}_y^1(\xi)}.
\end{multline*}
\end{theorem}

To help elucidate this theorem, we give some examples and degenerations.

\begin{example}
Suppose that $A_1(\xi)=\emptyset$ and $0 \notin A_{-1}(\xi),A_{-1}(\eta)$. The duality function then becomes
$$
1_{\{A_{-1}(\xi)\subseteq A_{-1}(\eta)\}} \prod_{y \in A_{-1}(\xi)} q^{2y + 2\cev{N}_y^{-1}(\eta)} \times q^{-2 \vert A_{-1}(\xi)\vert \cdot \vert A_1(\eta)\vert}.
$$
At infinitesimal times, both $A_{-1}(\xi)$ and $A_{-1}(\eta)$ evolve as ASEPs with right jump rates $1$ and left jump rates $q$. Furthermore, the term $q^{-2 \vert A_{-1}(\xi)\vert \cdot \vert A_1(\eta)\vert}$ remains constant, and we recover Sch\"{u}tz's duality result as a corollary.

\end{example}

\begin{example}
Now suppose that $A_{-1}(\xi)=\emptyset$ and $0 \notin A_{-1}(\xi),A_{-1}(\eta)$. The duality function then becomes
$$
1_{\{A_{-1}(\xi)\subseteq A_{-1}(\eta)\}}  
 \prod_{x \in A_1(\xi)} (Q^{-2}q^{-2x-2\vec{N}_x^1(\eta)})  \times q^{\vert A_1(\xi)\vert \cdot ( \vert A_1(\xi)\vert -1)} 
$$
\end{example}
At infinitesimal times, both $A_{1}(\xi)$ and $A_1(\eta)$ evolve an ASEPs with right jump rates $q$ and left jump rates $1$. Furthermore, the term $q^{\vert A_1(\xi)\vert \cdot ( \vert A_1(\xi)\vert -1)} $ remains constant, and we again recover Sch\"{u}tz's duality result, this time with the asymmetry direction reversed.

\begin{example}
The previous two examples had either $A_1(\xi),A_{-1}(\xi)$ as empty sets. Here, suppose that those sets are nonempty and interact with each other. The simplest case of this is when $A_1(\xi) = \{x\},A_{-1}(\xi)= \{x+1\}$ and $A_{-1}(\eta)= \{x\} , A_1(\eta) = \{x+1\}$, for some $x>0$. Then
$$
LD(\xi,\eta) = L(\xi,\eta) \cdot D(\eta,\eta) = q^2 \cdot Q^{-2}q^{-2(x+1)}q^{2x} \times 1 \times q^2 = Q^{-2}q^2
$$
and
$$
DL^T(\xi,\eta) = D(\xi,\xi)\cdot L(\eta,\xi) = 1 \cdot Q^{-2}q^{-2x}q^{2(x+1)} \times 1 \times 1 = Q^{-2}q^2.
$$
\end{example}

\begin{example}
Fix $L=3$ and define the particle configurations
$$
\eta = -1 \ -1 \ 1 \ 1, \quad \xi = 1 \ -1 \ 0 \ 1, \quad \hat{\eta}= 1 \ -1 \ 1 \ 1, \quad \hat{\xi}= -1 \ -1 \ 0 \ 1.
$$
Then
$$
LD(\eta,\xi) = L(\eta,\hat{\eta}) \cdot D(\hat{\eta},\xi) = Q^2 \cdot Q^{-2}q^{-0-4}Q^{-2}q^{-6-0}q^{2-0}\times q^{-2 \cdot 1 \cdot 3} \times q^{2\cdot 1}q^2 = Q^{-2}q^{-10}
$$
and
$$
DL(\eta,\xi) = D(\eta, \hat{\xi}) L(\xi,\hat{\xi}) = 1\cdot Q^{-2}q^{-6-0}q^{0-0}q^{2-2} \times q^{-2 \cdot 2 \cdot 2} \times q^{1\cdot 0} q^4 = Q^{-2}q^{-10}.
$$
\end{example}

\begin{example}
It is straightforward that the term $q^{-2 \vert A_{-1}(\xi)\vert \cdot \vert A_1(\eta)\vert} $ in the duality function only changes when a jump occurs at the boundary. Similarly, the term $q^{\vert A_1(\xi)\vert \cdot ( \vert A_1(\xi)\vert -1)} \prod_{y \in A_{-1}(\xi)} q^{2\vec{N}_y^1(\xi)}$ only changes when a $1$ and $-1$ switch places in the bulk. Indeed, if a $1$ particles jumps into the system from the boundary in $\xi$, the latter term is changed by 
$$
q^{ 2\vert A_1(\xi)\vert - 2\vec{N}_0^1(\xi)} =1.
$$
One can think of the duality function in the following way: the term in the first line of $D(\eta,\xi)$ is the ``naive'' extension of Sch\"{u}tz's duality result, the next term is the ``correction'' due to the boundary, and the last term is the ``correction'' due to interactions between $1$ and $-1$ particles in the bulk.
\end{example}

\subsubsection{Braided ASEP}
Let $\eta,\xi$ be two particle configurations of braided ASEP. Namely, suppose that $\eta(x)\in\{0,1,\ldots,m\}$ denotes the number of particles at lattice site $x$ in the configuration $\eta$, and similarly with $\xi$. Define the function $D_b(\eta,\xi)$ by
$$
D_b(\eta,\xi) = \prod_{x} \frac{ \binom{\eta_x}{\xi_x}_q}{\binom{m}{\xi_x}_q} q^{\xi_x(-2mx + 2\cev{N}_x(\eta))}
$$
Note that the duality functional is nonzero if and only if $\eta_x \geq \xi_x$ for all $x$. Up to space reversal, this is the same functional that appeared in \cite{CGRS} (and \cite{KIMRN} for the multi--species version); in \cite{CGRS},\cite{KIMRN} the particles drift to the right, whereas here the particles drift to the left. Note that the $q$--binomial in \cite{CGRS},\cite{KIMRN} is different from the one used here. When all $m=1$, this is the same duality function as in \eqref{MSch}, up to a normalizing constant.
\begin{theorem}\label{braided}
The braided ASEP is self--dual with respect to the function $D_b$.
\end{theorem}

\begin{example}
We will see that $LD(3\ 0, 1\ 1) = DL^T(3 \ 0 , 1 \ 1)$. The nonzero contributions are given by
\begin{align*}
L(3 \ 0, 2 \ 1) &= [3]_q(1-q^{2m})(1-q^{2m-2})q^{2(m-2)}, \quad \quad D(2\ 1, 1 \ 1) = \frac{\binom{2}{1}_q \binom{1}{1}_q}{\binom{m}{1}_q \binom{m}{1}_q}  q^{-2m+2\cdot 2} , \\
L(3 \ 0, 1 \ 2)&=[3]_q(1-q^{2m})q^{2\cdot (m-1)\cdot 2}, \quad \quad \quad \quad \quad \quad \quad D(1 \ 2, 1 \ 1)=\frac{\binom{1}{1}_q \binom{2}{1}_q}{\binom{m}{1}_q \binom{m}{1}_q} q^{-2m+2},\\
L( 1 \ 1, 2 \ 0)&=(1-q^{2m-2}) , \quad \quad \quad \quad\quad \quad \quad \quad \quad \quad \quad \quad D(3\ 0, 2\ 0)= \frac{\binom{3}{2}_q}{\binom{m}{2}_q}.
\end{align*}
So 
\begin{align*}
LD(3\ 0, 1 \ 1) &= L(3 \ 0 , 2 \ 1)D(2 \ 1, 1\ 1) + L(3 \ 0, 1 \ 2)D(1 \ 2, 1 \ 1) \\
&= \frac{[3]_q(1-q^{2m})[2]_q}{\binom{m}{1}_q\binom{m}{1}_q}\left(  (1-q^{2m-2})q^{2m-4}q^{-2m+4} + q^{4m-4} q^{-2m+2}\right) \\
&= \frac{[3]_q(1-q^{2m})[2]_q}{[m]_q[m]_q}
\end{align*}
and
\begin{align*}
DL^T(3 \ 0, 1 \ 1) &= D(3 \ 0 , 2 \ 0)L(1 \ 1, 2 \ 0)\\
&= \frac{[3]_q[2]_q}{[m]_q[m-1]_q}(1-q^{2m-2})\\
&=  \frac{[3]_q[2]_q}{[m]_q [m]_q } (1-q^{2m})
\end{align*}
so the two terms are equal.
\end{example}

\begin{example}
By a similar calculation,
$$
LD( 2 \ 4, 3 \ 1) = \frac{[2]_q[3]_q[4]_q}{[m]_q[m]_q[m-1]_q[m-2]_q}( q^{2m-8}[2]_q - (q^{2m-8}-1)[5]_q)
$$
and
$$
DL^T(2 \ 4, 3 \ 1) = \frac{[2]_q[3]_q[4]_q}{[m]_q[m]_q[m-1]_q[m-2]_q}( [m-2]_q(1-q^2)[3]_q+q^6[2]_q).
$$
which are equal to each other. 
\end{example}
\section{Background}

\subsection{Framework for Markov duality}\label{SETUP}
Recall the general set up of \cite{CGRS} and \cite{KIMRN}. Suppose that $M$ is a vector space with some canonical basis $\{c_b: b\in \mathcal{B}\}$.
\begin{itemize}
\item
There is a Hamiltonian $H$  such that $B^{-1}HB$ is self--adjoint for some diagonal matrix $B$; in other words $(B^{-1}HB)^T=B^{-1}HB.$ Further assume that the matrix entries of $H$ with respect to $c_b$ are all non--negative. 

\item
There exists an eigenvector $\Omega$  of $H$ with eigenvalues $\lambda$; in other words, $H\Omega = \lambda \Omega$. 

\item

Further suppose that $SH=HS$ for some symmetry $S$. Additionally suppose that every basis element $c_b$ occurs as a nonzero coefficient in $S^d\Omega$ for a unique $d\geq 0$.

\end{itemize}

Now define $G$ to be the diagonal matrix $G$ such that its $(b,b)$--entry is the nonzero coefficient of $c_b$ in $S^d\Omega$. Then $\mathcal{L}:=G^{-1}(H-\lambda {I})G$ is the generator of a continuous--time Markov chain. Additionally, $D:=G^{-1}SG^{-1}B^2$ is a self--duality function for the process generated by $\mathcal{L}$; that is, $\mathcal{L}D=D\mathcal{L}^T$. Additionally, $G^2B^{-2}$ defines reversible measures this process.

Additionally, suppose there is a matrix $V$ acting on $M$ such that $\mathcal{L}=V^{-1} \tilde{\mathcal{L}}V$, where $\tilde{\mathcal{L}}$ is the generator of another continuous--time process. Then $VD$ is a duality between $\mathcal{L}$ and $\mathcal{L}$. In other words, $\tilde{\mathcal{L}}VD=VD\mathcal{L}$.

\subsection{Hecke Algebras}

\subsubsection{Type A}
The Hecke algebra $\mathcal{H}_L(q)$ (of type $A_{L}$) is generated by $T_1,\ldots,T_{L}$ with relations
\begin{align*}
T_i T_{i+1} T_i &= T_{i+1}T_i T_{i+1}, \quad 1 \leq i \leq L,\\
T_i T_j &= T_j T_i, \quad \vert i-j\vert\geq 2,\\
T_i^2 &= (q-1)T_i + q T_1.
\end{align*}

Take $M$ to be (as a vector space) $\mathbb{C}^n \otimes \cdots \otimes \mathbb{C}^n$, with $L+1$ tensor powers. To define an $\mathcal{H}_L(q)$--module structure on $M$, it suffices to define the action of each $T_i$. Each $T_i$ acts on the $i$th and $(i+1)$--th tensor power $\mathbb{C}^n \otimes \mathbb{C}^n$ as the matrix $R$, where $R$ acts as
$$
\sum_{i<j} E_{j, i} \otimes E_{i, j}+q^{-1} \sum_{i<j} E_{i, j} \otimes E_{j, i}+\left(1- q^{-1}\right) \sum_{i<j} E_{j, j} \otimes E_{i, i}+\sum_{i} E_{i, i} \otimes E_{i, i},
$$
where $E_{i,j}$ denotes the matrix with a $1$ at the $(i,j)$--entry and $0$ elsewhere.

In \cite{BW}, there is a two parameter Hecke algebra $\mathcal{H}_L(r,s)$ with relations
$$
\begin{array}{l}\text { } T_{i} T_{i+1} T_{i}=T_{i+1} T_{i} T_{i+1}, \quad 1 \leq i\leq L \\ \text { } T_{i} T_{j}=T_{j} T_{i}, \quad|i-j| \geq 2 \\ \text {  } T_{i}^{2}=(s-r) T_{i}+r s 1\end{array}.
$$
The $R$--matrix is given by 
$$
r \sum_{i<j} E_{j, i} \otimes E_{i, j}+s^{-1} \sum_{i<j} E_{i, j} \otimes E_{j, i}+\left(1-r s^{-1}\right) \sum_{i<j} E_{j, j} \otimes E_{i, i}+\sum_{i} E_{i, i} \otimes E_{i, i}
$$
and thus there is a natural action on $M$.

\subsubsection{Type $B$}
The Hecke algebra $\mathfrak{H}_L$ (of type $B_L$) is generated by $T_0,T_1,\ldots,T_{k-1}$ subject to the relations
\begin{align*}
T_i T_{i+1}T_i &= T_{i+1}T_i T_{i+1}, \quad 1 \leq i < L\\
T_i T_j &= T_j T_i, \quad \vert i-j\vert\geq 2,\\
T_0T_1T_0T_1 &= T_1T_0T_1T_0,\\
(T_i - q^{-1})(T_i + q) &=0, \quad 1 \leq i < L\\
(T_0-Q^{-1})(T_0+Q) &=0.
\end{align*}
Note that this is the type $A_L$ Hecke algebra with an additional generator $T_0$.

Take $M$ to be (as a vector space) $\mathbb{C}^{2r+1} \otimes \cdots \otimes \mathbb{C}^{2r+1}$ with $L+1$ tensor powers. Each $T_i$ for $1 \leq i \leq L$ acts on the $i$th and $(i+1)$--th tensor power as 
$$
q^{-1} \sum_i E_{i,i} \otimes E_{i,i} + \sum_{i<j} E_{j,i} \otimes E_{i,j} + (q^{-1}-q)\sum_{i<j} E_{j,j} \otimes E_{i,i} + \sum_{i<j} E_{i,j} \otimes E_{j,i}.
$$
where the indices range over $-r,-r+1,\ldots,0,\ldots,r-1,r$. The generator $T_0$ acts on the first tensor power as
$$
Q^{-1} E_{0,0} + \sum_{i>0} E_{-i,i} + (Q^{-1}-Q)\sum_{i<0} E_{i,i} + \sum_{i<0} E_{-i,i}.
$$
In this way, $M$ is a representation of $\mathfrak{H}_L$.

A basis of $\mathbb{C}^{2r+1}$ is indexed by $\{-r,\ldots,0,\ldots,r\}$. A basis of $M$ is then indexed by sequences of integers $ b= (b_0,\ldots,b_{L})$ where $-r \leq b_l \leq r$. Let $\mathcal{B}$ denote the set of such sequences. From now on, let us fix $r=1$ for simplicity. It turns out to actually be easier to index $\mathcal{B}$ by sequences
\begin{equation}\label{seq}
(x_{d(1)},x_{d(1)-1},\ldots,x_1;z_0;y_1,\ldots,y_{d(-1)}),
\end{equation}
where $ d(1),d(-1) $ are non--negative integers such that $d(1) + d(-1) \leq L+1$, the $x$'s and $y$'s are ordered by $x_1<\ldots < x_{d(1)}$ and $y_1<\ldots, y_{d(-1)}$, and $z_0 \in \{-1,0,1\}$. The relationship between these two indices is that  $z_0=b_0$, each $b_{x_j}$ equals $1$ and each $b_{y_j}$ equals $-1$. In words, this means that $(x_1<\ldots < x_{d(1)})$ are exactly the locations of the particles labelled by $1$ that are away from site $0$, and $(y_1 < \ldots < y_{d(-1)})$ are exactly the locations of the particles labelled by $-1$ that are away from site $0$, and $z_0$ is the label of the particle located at site $0$. 

\subsubsection{Fused Hecke Algebra }
Before defining the fused Hecke algebra from \cite{crampe2020fused}, we first define fused permutations in terms of diagrams. Place two rows of $L$ ellipses, one on top of the other, and connect top ellipses with bottom ellipses with edges. Fix a sequence of positive integers $\mathbf{m}=(m_1,\ldots,m_L)$. For each $x \in \{1,\ldots,L\}$, exactly $m_a$ edges start from the $x$--th top ellipses and exactly $m_x$ edges arrive at the $x$--th bottom ellipse. For each $x \in \{1,\ldots,L\}$, let $I_a$ denote the multi-set indicating the bottom ellipses reached by the $m_x$ edges starting from the $x$--th top ellipse. Two diagrams are equivalent if their multisets $(I_1,\ldots,I_L)$ coincide. A \textit{fused permutation} is an equivalence class of diagrams. Below is an example of a fused permutation with the corresponding multi--set.

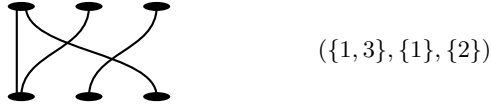
\begin{center}
\begin{tikzpicture}[scale=0.3]

\fill (21,2) ellipse (0.6cm and 0.2cm);\fill (21,-2) ellipse (0.6cm and 0.2cm);
\draw[thick] (20.8,2) -- (20.8,-2);\draw[thick] (21.2,2)..controls +(0,-2) and +(0,+2) .. (27,-2);  
\fill (24,2) ellipse (0.6cm and 0.2cm);\fill (24,-2) ellipse (0.6cm and 0.2cm);
\draw[thick] (24,2)..controls +(0,-2) and +(0,+2) .. (21,-2); 
\fill (27,2) ellipse (0.6cm and 0.2cm);\fill (27,-2) ellipse (0.6cm and 0.2cm);
\draw[thick] (27,2)..controls +(0,-2) and +(0,+2)..(24,-2);
\node at (38,0) {$(\{1,3\},\{1\},\{2\})$};

\end{tikzpicture}
\end{center}
If all $m_a=1$, the a fused permutation is just a usual permutation in the symmetric group $S_L$.

Fused permutations can be multiplied by concatenating two diagrams, removing the middle ellipses, and normalizing. Because the precise form of the multiplication will not be needed here, we refer to \cite{crampe2020fused} for details. 

\begin{definition}
The $\mathbb{C}$-vector space $H_{\bk,L}(q)$ is the quotient of the vector space with basis indexed by fused braids by the following relations:
 \begin{itemize}
   \item[(i)] The Hecke relation:  
   \begin{center}
 \begin{tikzpicture}[scale=0.25]
\draw[thick] (0,2)..controls +(0,-2) and +(0,+2) .. (4,-2);
\fill[white] (2,0) circle (0.4);
\draw[thick] (4,2)..controls +(0,-2) and +(0,+2) .. (0,-2);
\node at (6,0) {$=$};
\draw[thick] (12,2)..controls +(0,-2) and +(0,+2) .. (8,-2);
\fill[white] (10,0) circle (0.4);
\draw[thick] (8,2)..controls +(0,-2) and +(0,+2) .. (12,-2);
\node at (17,0) {$-\,(q-q^{-1})$};
\draw[thick] (21,2) -- (21,-2);\draw[thick] (25,2) -- (25,-2);
\end{tikzpicture}
\end{center}
  \item[(ii)]  The idempotent relations: for top ellipses,
 \begin{center}
 \begin{tikzpicture}[scale=0.4]
\fill (2,2) ellipse (0.8cm and 0.2cm);
%\draw[thick] (1.6,2)..controls +(0,-1) and +(1,1) .. (-1,0);
\draw[thick] (2.2,2)..controls +(0,-1.5) and +(1,1) .. (1.2,0);
\fill[white] (2,0.7) circle (0.2);
\draw[thick] (1.8,2)..controls +(0,-1.5) and +(-1,1) .. (2.8,0);
%\draw[thick] (2.4,2)..controls +(0,-1) and +(-1,1.4) .. (4.2,0);
\node at (4.5,1) {$=$};
\node at (7,1) {$q$};
\fill (9,2) ellipse (0.8cm and 0.2cm);
%\draw[thick] (8.6,2)..controls +(0,-1) and +(1,1) .. (6,0);
\draw[thick] (8.8,2)..controls +(0,-1.5) and +(0.5,0.5) .. (8.2,0);
\draw[thick] (9.2,2)..controls +(0,-1.5) and +(-0.5,0.5) .. (9.8,0);
%\draw[thick] (9.4,2)..controls +(0,-1) and +(-1,1.4) .. (11.2,0);

\node at (13,1) {and};

\fill (18,2) ellipse (0.8cm and 0.2cm);
%\draw[thick] (17.6,2)..controls +(0,-1) and +(1,1) .. (15,0);
\draw[thick] (17.8,2)..controls +(0,-1.5) and +(-1,1) .. (18.8,0);
\fill[white] (18,0.7) circle (0.2);
\draw[thick] (18.2,2)..controls +(0,-1.5) and +(1,1) .. (17.2,0);
%\draw[thick] (18.4,2)..controls +(0,-1) and +(-1,1.4) .. (20.2,0);
\node at (20.5,1) {$= $};
\node at (23,1) {$q^{-1}$};
\fill (25,2) ellipse (0.8cm and 0.2cm);
%\draw[thick] (24.6,2)..controls +(0,-1) and +(1,1) .. (22,0);
\draw[thick] (24.8,2)..controls +(0,-1.5) and +(0.5,0.5) .. (24.2,0);
\draw[thick] (25.2,2)..controls +(0,-1.5) and +(-0.5,0.5) .. (25.8,0);
%\draw[thick] (25.4,2)..controls +(0,-1) and +(-1,1.4) .. (27.2,0);
\end{tikzpicture}
\end{center}
and for bottom ellipses,
\begin{center}
 \begin{tikzpicture}[scale=0.4]
\fill (2,0) ellipse (0.8cm and 0.2cm);
%\draw[thick] (1.6,0)..controls +(0,1) and +(1,-1.5) .. (-1,2);
\draw[thick] (1.8,0)..controls +(0,1.5) and +(-1,-1) .. (2.8,2);
\fill[white] (2,1.25) circle (0.2);
\draw[thick] (2.2,0)..controls +(0,1.5) and +(1,-1) .. (1.2,2);
%\draw[thick] (2.4,0)..controls +(0,1) and +(-1,-1.4) .. (4.2,2);
\node at (4.5,1) {$=$};
\node at (7,1) {$q$};
\fill (9,0) ellipse (0.8cm and 0.2cm);
%\draw[thick] (8.6,0)..controls +(0,1) and +(1,-1.5) .. (6,2);
\draw[thick] (9.2,0)..controls +(0,1.5) and +(-0.5,-0.5) .. (9.8,2);
\draw[thick] (8.8,0)..controls +(0,1.5) and +(0.5,-0.5) .. (8.2,2);
%\draw[thick] (9.4,0)..controls +(0,1) and +(-1,-1.4) .. (11.2,2);

\node at (13,1) {and};

\fill (18,0) ellipse (0.8cm and 0.2cm);
%\draw[thick] (17.6,0)..controls +(0,1) and +(1,-1.5) .. (15,2);
\draw[thick] (18.2,0)..controls +(0,1.5) and +(1,-1) .. (17.2,2);
\fill[white] (18,1.25) circle (0.2);
\draw[thick] (17.8,0)..controls +(0,1.5) and +(-1,-1) .. (18.8,2);
%\draw[thick] (18.4,0)..controls +(0,1) and +(-1,-1.4) .. (20.2,2);
\node at (20.5,1) {$=$};
\node at (23,1) {$q^{-1}$};
\fill (25,0) ellipse (0.8cm and 0.2cm);
%\draw[thick] (24.6,0)..controls +(0,1) and +(1,-1.5) .. (22,2);
\draw[thick] (25.2,0)..controls +(0,1.5) and +(-0.5,-0.5) .. (25.8,2);
\draw[thick] (24.8,0)..controls +(0,1.5) and +(0.5,-0.5) .. (24.2,2);
%\draw[thick] (25.4,0)..controls +(0,1) and +(-1,-1.4) .. (27.2,2);
\end{tikzpicture}
\end{center}
 \end{itemize}
 \end{definition}

In fact, we will not be using the relations in this paper either, but they are included for completeness. Here, we will primarily be focused on the elements $\Sigma_x$, which moves all strands starting from ellipse $x$ to ellipse $x+1$, and vice versa. Below is an example when all $m_x=2$. 
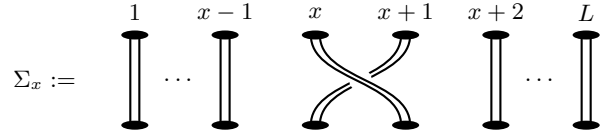
\begin{center}
 \begin{tikzpicture}[scale=0.3]
\node at (-2,0) {$\Sigma_x:=$};
\node at (2,3) {$1$};\fill (2,2) ellipse (0.6cm and 0.2cm);\fill (2,-2) ellipse (0.6cm and 0.2cm);
\draw[thick] (1.8,2) -- (1.8,-2);\draw[thick] (2.2,2) -- (2.2,-2);
\node at (4,0) {$\dots$};
\draw[thick] (5.8,2) -- (5.8,-2);\draw[thick] (6.2,2) -- (6.2,-2);
\node at (6,3) {$x-1$};\fill (6,2) ellipse (0.6cm and 0.2cm);\fill (6,-2) ellipse (0.6cm and 0.2cm);
\node at (10,3) {$x$};\fill (10,2) ellipse (0.6cm and 0.2cm);\fill (10,-2) ellipse (0.6cm and 0.2cm);
\node at (14,3) {$x+1$};\fill (14,2) ellipse (0.6cm and 0.2cm);\fill (14,-2) ellipse (0.6cm and 0.2cm);

\draw[thick] (13.8,2)..controls +(0,-2) and +(0,+2) .. (9.8,-2);
\draw[thick] (14.2,2)..controls +(0,-2) and +(0,+2) .. (10.2,-2);
\fill[white] (12,0) circle (0.5);
\draw[thick] (10.2,2)..controls +(0,-2) and +(0,+2) .. (14.2,-2);
\draw[thick] (9.8,2)..controls +(0,-2) and +(0,+2) .. (13.8,-2);

\draw[thick] (17.8,2) -- (17.8,-2);\draw[thick] (18.2,2) -- (18.2,-2);
\node at (18,3) {$x+2$};\fill (18,2) ellipse (0.6cm and 0.2cm);\fill (18,-2) ellipse (0.6cm and 0.2cm);
\node at (20,0) {$\dots$};
\draw[thick] (21.8,2) -- (21.8,-2);\draw[thick] (22.2,2) -- (22.2,-2);\fill (22,2) ellipse (0.6cm and 0.2cm);\fill (22,-2) ellipse (0.6cm and 0.2cm);
\node at (22,3) {$L$};
\end{tikzpicture}
\end{center}
These elements satisfy the braid relation $\Sigma_x\Sigma_{x+1}\Sigma_x = \Sigma_{x+1}\Sigma_x\Sigma_{x+1}$. 

When all $m_x=1$, define the $R$--matrix $\check{R}$ 
from (12) of \cite{crampe2020fused} by 
$$
\check{R}\left(e_{i} \otimes e_{j}\right):=\left\{\begin{array}{ll}q e_{i} \otimes e_{j} & \text { if } i=j \\ e_{j} \otimes e_{i}+\left(q-q^{-1}\right) e_{i} \otimes e_{j} & \text { if } i<j, \quad \text { where } i, j=1, \ldots, N \\ e_{j} \otimes e_{i} & \text { if } i>j\end{array}\right.
$$
We write this as
$$
\check{R} = \sum_{i>j} E_{j,i} \otimes E_{i,j} + \sum_{i<j} E_{j,i} \otimes E_{i,j}+(q-q^{-1})\sum_{i<j} E_{i,i} \otimes E_{j,j}+ q\sum_i  E_{i,i} \otimes E_{i,i}
$$

\subsection{Quantum groups}

\subsubsection{A two--parameter quantum group}

We recall the two--parameter quantum group from \cite{BW}. 
Let $\alpha_i = \epsilon_{i-1/2} - \epsilon_{i+1/2}$ denote the simple roots in the root system of type $A_{n}$.
Let $\mathcal{U}_{r,s}(\mathfrak{sl}_{n+1})$ be generated by $E_i,F_i,\omega_i^{\pm 1},\omega_i'^{\pm 1}$ for $i \in \{1,2,\ldots,n\}$, with relations
$$
\begin{aligned} \omega_{{i}} \omega_{{i}}^{-1} &=\omega_{{i}}^{-1} \omega_{{i}}=1 \\ 
\omega_{{i}} \omega_{{j}} &=\omega_{{j}} \omega_{{i}} \\
\omega_{i} e_{j}&=r^{\left\langle\epsilon_{i}, \alpha_{j}\right\rangle} s^{\left\langle\epsilon_{i+1}, \alpha_{j}\right\rangle} e_{j} \omega_{i}\\ 
 \omega_{i}^{\prime} e_{j}&=r^{\left\langle\epsilon_{i+1}, \alpha_{j}\right\rangle} s^{\left\langle\epsilon_{i}, \alpha_{j}\right\rangle} e_{j} \omega_{i}^{\prime}\\ 
\omega_{i} f_{j} &=r^{-\left\langle\epsilon_{i}, \alpha_{j}\right\rangle} s^{-\left\langle\epsilon_{i+1}, \alpha_{j}\right\rangle} f_{j} \omega_{i} \\ \omega_{i}^{\prime} f_{j} &=r^{-\left\langle\epsilon_{i+1}, \alpha_{j}\right\rangle} s^{-\left\langle\epsilon_{i}, \alpha_{j}\right\rangle} f_{j} \omega_{i}^{\prime} \\
 E_{{i}} F_{{j}}-F_{{j}} E_{{i}} &=\frac{\delta_{i, j}}{r-s}\left(\omega_{i}-\omega_{i}^{\prime}\right)\\
 E_{{i}}^{2} E_{{j}}+rsE_{{j}} E_{{i}}^{2} &=\left(r+s\right) E_{{i}} E_{{j}} E_{{i}}, & \text { if }i-j=1 \\ 
  rsE_{{i}}^{2} E_{{j}}+E_{{j}} E_{{i}}^{2} &=\left(r+s\right) E_{{i}} E_{{j}} E_{{i}}, & \text { if }i-j=-1 \\ 
 E_{{i}} E_{{j}} &=E_{{j}} E_{{i}}, & \text { if }|i-j|>1 \\ 
 F_{{i}}^{2} F_{{j}}+ r^{-1}s^{-1}F_{{j}} F_{{i}}^{2}&=\left(r^{-1}+s^{-1}\right) F_{{i}} F_{{j}} F_{{i}}, & \text { if }i-j=1 \\
  r^{-1}s^{-1}F_{{i}}^{2} F_{{j}}+ F_{{j}} F_{{i}}^{2}&=\left(r^{-1}+s^{-1}\right) F_{{i}} F_{{j}} F_{{i}}, & \text { if }i-j=-1 
 \\ F_{{i}} F_{{j}} &=F_{{j}} F_{{i}}, &|i-j|>1 \end{aligned}
$$ 
and comultiplication 
$$
\Delta\left(e_{i}\right)=e_{i} \otimes 1+\omega_{i} \otimes e_{i}, \quad \Delta\left(f_{i}\right)=1 \otimes f_{i}+f_{i} \otimes \omega_{i}^{\prime}, \quad \Delta(\omega_i)=\omega_i \otimes \omega_i, \quad \Delta(\omega_i')=\omega_i' \otimes \omega_i'.
$$
When $r=q=s^{-1}$, the algebra $\mathcal{U}_{r,s}(\mathfrak{sl}_{n+1})$ modulo the ideal generated by $\omega_j'-\omega_j^{-1}$ is the usual Drinfeld--Jimbo $\mathcal{U}_q(\mathfrak{sl}_{n+1})$. The element $\omega_j \in \mathcal{U}_q(\mathfrak{sl}_{n+1})$ is more commonly denoted $K_j$.

\subsubsection{Co--ideal subalgebra}
Consider $\mathcal{U}_q(\mathfrak{sl}_{n+1})$, and suppose that $n=2r$ is even. In this case, instead let $i$ range over the values $-r+1/2,-r+3/2,\ldots,-1/2,1/2,\ldots,r-1/2$.

Define a subalgebra $\mathbf{U}$ which is generated by (for $i=1/2,3/2,\ldots,r-1/2)$
$$
\begin{array}{cl}k_{i}=K_{i} K_{-i}^{-1}, & e_{i}=E_{i}+F_{-i} K_{i}^{-1}\left(i \neq \frac{1}{2}\right), \quad f_{i}=E_{-i}+K_{-i}^{-1} F_{i}\left(i \neq \frac{1}{2}\right) \\ & e_{\frac{1}{2}}=E_{\frac{1}{2}}+Q^{-1} F_{-\frac{1}{2}} K_{\frac{1}{2}}^{-1}, \quad f_{\frac{1}{2}}=E_{-\frac{1}{2}}+Q K_{-\frac{1}{2}}^{-1} F_{\frac{1}{2}}\end{array}
$$
%The algebra $\mathbf{U}$ has a presentation 
%$$
%\begin{aligned} k_{i} k_{i}^{-1} &=k_{i}^{-1} k_{i}=1, \quad k_{i} k_{j}=k_{j} k_{i} \\ k_{i} e_{j} k_{i}^{-1} &=q^{\left(\alpha_{i}-\alpha_{-i}, \alpha_{j}\right)} e_{j}, \quad k_{i} f_{j} k_{i}^{-1}=q^{-\left(\alpha_{i}-\alpha_{-i}, \alpha_{j}\right)} f_{j} \\ e_{i} f_{j}-f_{j} e_{i} &=\delta_{i, j} \frac{k_{i}-k_{i}^{-1}}{q-q^{-1}} \\ e_{i} e_{j} &=e_{j} e_{i}, \quad f_{i} f_{j}=f_{j} f_{i} \\ e_{i}^{2} e_{j}+e_{j} e_{i}^{2} &=\left(q+q^{-1}\right) e_{i} e_{j} e_{i}, \quad f_{i}^{2} f_{j}+f_{j} f_{i}^{2}=\left(q+q^{-1}\right) f_{i} f_{j} f_{i}, \quad \text { if }|i-j|>1 \\ e_{\frac{1}{2}}^{2} f_{\frac{1}{2}}+f_{\frac{1}{2}} e_{\frac{1}{2}}^{2} &=\left(q+q^{-1}\right)\left(e_{\frac{1}{2}} f_{\frac{1}{2}} e_{\frac{1}{2}}-e_{\frac{1}{2}}\left(p q k_{\frac{1}{2}}+p^{-1} q^{-1} k_{\frac{1}{2}}^{-1}\right)\right) \\ f_{\frac{1}{2}}^{2} e_{\frac{1}{2}}+e_{\frac{1}{2}} f_{\frac{1}{2}}^{2} &=\left(q+q^{-1}\right)\left(f_{\frac{1}{2}} e_{\frac{1}{2}} f_{\frac{1}{2}}-\left(p q k_{\frac{1}{2}}+p^{-1} q^{-1} k_{\frac{1}{2}}^{-1}\right) f_{\frac{1}{2}}\right) \end{aligned}
%$$
In order to make the notation consistent with \cite{BWW}, we will consider the co--product
$$
\Delta\left(E_{i}\right)=1 \otimes E_{i}+E_{i} \otimes K_{i}^{-1}, \Delta\left(F_{i}\right)=F_{i} \otimes 1+K_{i} \otimes F_{i}, \Delta\left(K_{i}\right)=K_{i} \otimes K_{i}
$$
Note that $\mathbf{U}$ is a right coideal subalgebra, meaning that $\Delta(\mathbf{U}) \subset \mathbf{U} \otimes \mathcal{U}_q(\mathfrak{sl}_{2r+1})$. Indeed, 
\begin{align*}
\Delta(e_i) &= e_i \otimes K_i^{-1} + 1 \otimes E_i + Q^{-1} k_{-i} \otimes F_{-i}K_i^{-1} ,\\
\Delta(f_i) &= f_i \otimes K_{-i}^{-1} + 1\otimes E_{-i} + Q k_i \otimes K_{-i}^{-1}F_i, \text{ for } i\neq \tfrac{1}{2}.
\end{align*}
In particular, 
\begin{multline*}
\Delta^{(L)} (f_{1/2}) = f_{1/2} \otimes (K_{-1/2}^{-1} )^{\otimes L} + 1 \otimes \sum_{j=1}^{L} 1^{\otimes j-1} \otimes E_{-1/2} \otimes (K_{-1/2}^{-1})^{\otimes L-j} \\
+Q k_{1/2} \otimes \sum_{j=1}^{L} (k_{1/2})^{\otimes j-1} \otimes K_{-1/2}^{-1}F_{1/2} \otimes (K_{-1/2}^{-1})^{\otimes L-j}
\end{multline*}

The natural representation of $\mathcal{U}_q(\mathfrak{sl}_{2r+1})$ has a basis indexed by $u_{-r},\ldots,u_0\ldots,u_r$. Define the action so that $E_{k-1/2}$ maps $u_k$ to $u_{k-1}$ and $F_{k+1/2}$ maps $u_k$ to $u_{k+1}$. Likewise $F_{k-1/2}$ maps $u_{k-1}$ to $u_k$ and $E_{k+1/2}$ maps $u_{k+1}$ to $u_k$.

There are also representations of $\mathcal{U}_q(\mathfrak{sl}_{n+1})$ on $\mathrm{Sym}^m \mathbb{C}^{n+1}$ for each $m\geq 1$, but we will not need an explicit expression for this representation.

\subsection{Schur--Weyl Duality}
With the Hecke algebra and the quantum groups defined, we proceed to state the Schur--Weyl duality that will be needed to show Markov duality. Roughly speaking, Schur--Weyl duality states that on certain representations, the actions of the Hecke algebra and quantum groups form double centralizers. Additionally, Schur--Weyl duality gives a decomposition of the representations into irreducibles, but we will not need that result. Here, only a weaker statement is needed, namely that the actions commute with each other. 

\subsubsection{A two parameter Schur--Weyl duality}
The action of $\mathcal{H}_L(r,s)$ and the action of $\mathcal{U}_{r,s}(\mathfrak{sl}_{n+1})$ commute on $(\mathbb{C}^{n+1} )^{\otimes L}$.

\subsubsection{Type B Hecke algebra and co--ideal subalgebra}
The action of $\mathfrak{H}_L$ and the action of the co--ideal subalgebra $\mathbf{U}$ commute on $(\mathbb{C}^{2r+1} )^{\otimes L}$ for all $r \geq 1$. Note that we are requiring the dimension of $\mathbb{C}^{2r+1}$ to be odd. There is an even version of this statement, but does not seem to lead to nice duality functions. (See \cite{KOpen}).

\subsubsection{Fused Schur--Weyl Duality}
For any $\mathbf{m}=(m_1,\ldots,m_L)$, the action of $H_{\mathbf{m},L}(q)$ and the action of $\mathcal{U}_q(\mathfrak{sl}_{n+1})$ on $\mathrm{Sym}^{m_1} \mathbb{C}^{n+1} \otimes \cdots \otimes \mathrm{Sym}^{m_L} \mathbb{C}^{n+1} $ commute with each other.

\section{Proofs}
First, let us outline what is needed for each proof. It is important to note that the matrix $G$ has already been calculated in \cite{KIMRN}, so this step does not need to be repeated. 

For Theorem \ref{Sch}, it is clear from conjugating by $G$ that the dynamics are the multi--species ASEP. What remains is to check that the same duality function \eqref{Sch} arises from the action of the two--parameter quantum group. 

For Theorem \ref{OPEN}, it is again clear from conjugating by $G$ that the dynamics are the open multi--species ASEP described at the beginning of this paper. What remains is to find the duality function that arises from the co--ideal subalgebra $\mathbf{U}$. This duality function will be new. 

For Theorem \ref{braided}, the duality function is already known from \cite{CGRS}. This time, the primary work comes in giving explicit formulas for the jump rates that comes from the Hecke algebra action.

\subsection{Proof of Theorem \ref{Sch} }
Because this result is not new, we keep the proof brief. We actually only check it for $n=1$; the case for higher values of $n$ is similar. 

In order for the dynamics to be stochastic, we must have $s=q,r=q^{-1}$. What remains is to check the duality function. We write the duality function for general values of $r,s$ and plug in their values at the end.

Write $\eta = (\eta_x)_{1 \leq x \leq L}$ where $\eta_x=(\eta_x^{(1)},\eta_x^{(2)})$ is either $(1,0)$ or $(0,1)$.
By the pseudo--factorization property of \cite{CGRS},
$$
S(\eta,\xi) = 1_{\eta \geq \xi} \prod_{\substack{ \xi_x = (0,1) \\ \eta_x=(1,0)}} \left( r^{-1}s \cdot s^{N_x^{(1)}(\xi)}r^{N_x^{(2)}(\xi)}\right).
$$
Here,
$$
N_x^{(i)}(\xi) = \sum_{y>x} \xi_y^{(i)}.
$$
We also have
$$
G(\xi,\xi) = \prod_{\xi_x=(0,1)} (r^{-1}s \cdot s^{N_x^{(1)}(\xi)}r^{N_x^{(2)}(\xi)}) 
$$
Let 
$$
C_1(\eta,\xi) =  \prod_{\substack{ \xi_x = (0,1) \\ \eta_x=(1,0)}} r^{-1}s, \quad \quad C_2(\xi) = \prod_{\xi_x=(0,1)} (r^{-1}s  \cdot r^{N_x^{(2)}(\xi)}) .
$$
There are constants if the particle number is conserved. Then 
$$
G(\xi,\xi) = C_2(\xi)\prod_{x<y} s^{\xi^{(2)}_x \xi^{(1)}_y},
$$
so
\begin{align*}
SG^{-1}(\eta,\xi) = C_1(\eta,\xi) C_2(\xi)^{-1}1_{\eta \geq \xi} & \prod_{x<y} (s^{\eta_x^{(1)} \xi_x^{(2)}\xi_y^{(1)}}r^{\eta_x^{(1)} \xi_x^{(2)}\xi_y^{(2)}} ) \prod_{x<y}s^{-\xi^{(2)}_x \xi^{(1)}_y}\\
C_1(\eta,\xi) C_2(\xi)^{-1}1_{\eta \geq \xi} &\prod_{x<y} (s^{-\eta_x^{(2)} \xi_x^{(2)}\xi_y^{(1)}}r^{\eta_x^{(1)} \xi_x^{(2)}\xi_y^{(2)}} )
\end{align*}
where we used that $\eta_x^{(1)}-1=-\eta_x^{(2)}$. Thus,
$$
G^{-1}SG^{-1}(\eta,\xi) = C_2(\eta)^{-1}C_1(\eta,\xi) C_2(\xi)^{-1} 1_{\eta \geq \xi}  \prod_{x<y} (s^{-\eta_x^{(2)}( \xi_x^{(2)}\xi_y^{(1)} +  \eta_y^{(1)} )}r^{\eta_x^{(1)} \xi_x^{(2)}\xi_y^{(2)}} ).
$$
Note that on the set $\{\eta \geq \xi\}$, if $\eta_x^{(2)}=1$ then $\xi_x^{(2)}=1$. So, again using that $\eta^{(1)}_x = 1 - \eta_x^{(2)}$,
$$
D(\eta,\xi) = C_2(\eta)^{-1}C_1(\eta,\xi) C_2(\xi)^{-1} C_3(r,\xi) 1_{\eta \geq \xi}  (s^{-\eta_x^{(2)}( \xi_y^{(1)} +  \eta_y^{(1)} )}r^{-\eta_x^{(2)} \xi_y^{(2)}} ),
$$
where 
$$
C_3(r,\xi) = \prod_{x<y} r^{\xi_x^{(2)}\xi_y^{(2)}}.
$$
We can now write the duality in terms of sites where $\eta_x^{(2)},\xi_x^{(2)}=1$. We have
$$
D(\eta,\xi) =  C_2(\eta)^{-1}C_1(\eta,\xi) C_2(\xi)^{-1} C_3(r,\xi) C_3(s,\eta) 1_{\eta \geq \xi} \prod_{\eta_x^{(2)}=1} (s^{-2})^{L-x} s^{N_x^{(2)}(\xi)}r^{-N_x^{(2)}(\xi)}.
$$
Note that if $s=q,r=q^{-1}$, we recover Sch\"{u}tz's duality functional. Note that the constant simplifies as
$$
C_2(\eta)^{-1} \left( \prod_{\xi_x=\eta_x=(0,1)} rs^{-1} \right)  C_3(s,\eta). 
$$

%The $R$--matrix is given by 
%$$
%\left( 
%\begin{array}{cccc}
%1 & 0 & 0 & 0 \\
%0 & 0 & s^{-1} & 0 \\
%0 & r & 1-rs^{-1} & 0 \\
%0 & 0 & 0 & 1
%\end{array}
%\right),
%$$
%and the conjugated stochastic $S$--matrix is
%$$
%\left( 
%\begin{array}{cccc}
%1 & 0 & 0 & 0 \\
%0 & 0 & 1 & 0 \\
%0 & rs^{-1} & 1-rs^{-1} & 0 \\
%0 & 0 & 0 & 1
%\end{array}
%\right).
%$$
%Note that
%$$
%(1 \ r \ s \ 1)S = (1\ r \ s \ 1).
%$$

\subsection{Proof of Theorem \ref{OPEN}}

Define 
\begin{align*}
a_0 &= f_{1/2} \otimes (K_{-1/2}^{-1} )^{\otimes L} \\
a_j^+ &= 1^{\otimes j} \otimes E_{-1/2} \otimes (K_{-1/2}^{-1})^{\otimes L-j} ,\\
a_j^- &= Q (k_{1/2})^{\otimes j} \otimes K_{-1/2}^{-1}F_{1/2} \otimes (K_{-1/2}^{-1})^{\otimes L-j}.
\end{align*}
In particular, 
$$
\Delta^{(L)}(f_{1/2}) = a_0 + \sum_{y=1}^L a_y^+ + \sum_{x=1}^L a_x^-.
$$
These terms commute with each other in a nice way, as seen in the next lemma.  
\begin{lemma}
For $l<j$,
$$
a^+_j a^+_l = q^2 a^+_l a^+_j, \quad a^-_l a^-_j = q^2 a^-_j a^-_l, \quad a^+_l a^-_j = q^2 a^-_j a^+_l, \quad a^+_j a^-_l = q^2 a^-_l a^+_j, \quad a_j^+a_0 = q^2 a_0 a_j^+, \quad a_0 a_j^- = q^2 a_j^-a_0.
$$
\end{lemma}
\begin{proof}
The first (and fourth and fifth) identity follows from $E_{-1/2}K_{-1/2}^{-1} = q^2 K_{-1/2}^{-1}E_{-1/2}$. The second identity follows from $K_{-1/2}^{-1}F_{1/2}=q^{-1} F_{1/2}K_{-1/2}^{-1}$ and $F_{1/2}k_{1/2}=q^3k_{1/2}F_{1/2}.$ The third identity follows from $K_{-1/2}^{-1}F_{1/2} = q^{-1}F_{1/2}K_{-1/2}^{-1}$ and $E_{-1/2}k_{1/2}=q^3 k_{1/2}E_{-1/2}$. The sixth identity follows from $f_{1/2}k_{1/2}=q^3k_{1/2}f_{1/2}$ and $K_{-1/2}^{-1}F_{1/2}=q^{-1} F_{1/2}K_{-1/2}^{-1}$.
\end{proof}

Let $\mathcal{B}^0_d\subset \mathcal{B}$ denote the set of sequences \eqref{seq} such that $d(1) + d(-1) = d$ and $z_0=0$. Define $\mathcal{B}^1_d$ and $\mathcal{B}^{-1}_d$ in a similar manner. Let $b_0 \in \mathcal{B}$ denote the sequence given by $(0)$. In other words, there are only particles labeled by $0$.
\begin{lemma}
For every $d \geq 0$,
\begin{multline*}
([d]_{q^2}^! )^{-1} \Delta^{(L)}(f_{1/2}^d) c_{b_0}  = \sum_{b \in \mathcal{B}_d^0} a_{y_{d(-1)}}^+\cdots a_{y_{1}}^+ a_{x_1}^- \cdots a_{x_{d(1)}}^-  c_b+ \sum_{b \in \mathcal{B}_{d-1}^1} a_{y_{d(-1)}}^+\cdots a_{y_{1}}^+ a_0 a_{x_1}^- \cdots a_{x_{d(1)}}^- c_b\\
+ \sum_{b \in \mathcal{B}_{d-1}^{-1}}  a_{y_{d(-1)}}^+\cdots a_{y_{1}}^+ a_0 a_{x_1}^- \cdots a_{x_{d(1)}}^- c_b.
\end{multline*}
In the summations, $d(1)$ and $d(-1)$ depend implicitly on the index $b$. 
\end{lemma}
\begin{proof}
By definition of the $a_0,a^{\pm}$, 
$$
([d]_{q^2}^! )^{-1}\Delta^{(L)}(f_{1/2}^d) = ([d]_{q^2}^! )^{-1}\left( a_0 + \sum_{y=1}^L a_y^+ + \sum_{x=1}^L a_x^-\right)^d.
$$
Acting on $M$, we have $(a_0)^2 = (a_y^+)^2 = (a_x^-)^2=0$. Therefore, when expanding the product, there are $d(1)$ distinct terms $a^-_{x_1},\ldots, a^-_{x_{d(1)}}$ and $d(-1)$ distinct terms $a^+_{y_1},\ldots, a^+_{y_{d(-1)}}$. The summation over terms without $a_0$ then becomes a summation over $\mathcal{B}_d^0$. By the previous lemma, this summation equals the first sum on the right--hand--side. Similarly, the summation over terms with $a_0$ becomes a summation over $\mathcal{B}_{d-1}^1$ and $\mathcal{B}_{d-1}^{-1}$. By the previous lemma, the summation equals the second and third sums respectively on the right--hand--side. 

\end{proof}

Taking $c_{b_0}$ to be the eigenvector $\Omega$ and $S$ to be the symmetry $\Delta^{(L-1)}(f_{1/2})$, we see that the second and third bullet points in section \ref{SETUP} hold. For $q,Q \in (0,1)$, the first bullet point holds, with $H$ already being self--adjoint.  The next step is to calculate the action of $S^d$ for $d\geq 1$. As a warmup, let us first find $G$. For any $1 \leq y \leq L$, let 
$$
N_y^1(b) = \left| \{ x_j: x_j > y\} \right|.
$$
In words, this is the number of particles labeled by $1$ that are positioned to the right of $y$. 

\begin{prop}
(a) For $b$ given as in \eqref{seq}, the matrix $G$ is given by 
$$
G(b,b) = \left[\prod_{j=1}^{d(1)} Q q^{L+x_j-\vec{N}_{x_j}^1(b)} \right] \zeta_0 \left[\prod_{i=1}^{d(-1)} q^{L-y_i-\vec{N}^1_{y_i}(b)}\right],
$$
where
$$
\zeta_0
= 
\begin{cases}
Qq^{L-d(1)}, &\text{ if } z_0=1\\
1, &\text{ if } z_0=0\\
q^{L-d(1)}, &\text{ if } z_0=-1.
\end{cases}
$$

(b) The operator $\mathcal{L}=G^{-1}(H-\lambda I)G$ is the generator of the open multi--species ASEP.

(c) In particular, the measure $\pi(b) = G(b,b)^2$ is reversible for the process. 
\end{prop}
\begin{proof}
(a) By the previous lemma, we need only apply the $a^-_x$ in order, then $a_0$ (if needed), then $a^+_y$. In the application of $a_x^-$, the $k_{1/2}$ term on the left produces $q^{2x_j}$, while the $K_{-1/2}^{-1}$ on the right produces $q^{L-x_j-(d(1)-j)}$. The $-(d(1)-j)$ term comes from not counting the particles located at $x_{j+1},\ldots,x_{d(1)}$, and by definition equals $-N_{x_j}^1(b)$. The $a_0$ term gives $\zeta_0$. Finally, the $a^+_y$ only has a term on the right, which contributes $q^{L-y_j-N^1_{y_j}(b)}$; as before, the $-N_{y_i}^1(b)$ comes from not counting the particles labeled by $1$ that are to the right of $y_i$. 

(b) We need only check the action of each local Hamiltonian $T_i$. By direct inspection, for $1 \leq x <L$, each $qT_x$ acts as 
$$
 \sum_i E_{i,i} \otimes E_{i,i} +q^2 \sum_{i<j} E_{j,i} \otimes E_{i,j} + (1-q^2)\sum_{i<j} E_{j,j} \otimes E_{i,i} +  \sum_{i<j} E_{i,j} \otimes E_{j,i},
$$
which is stochastic. Similarly, $QT_0$ acts as 
$$
 E_{0,0} + Q^2 \sum_{i>0} E_{-i,i} + (1-Q^2)\sum_{i<0} E_{i,i} +  \sum_{i<0} E_{-i,i},
$$
which is also stochastic. 

(c) This follows from the discussion in section \ref{SETUP}.
\end{proof}

\begin{remark}
Setting $x_0$ to be $0$ if $\zeta_0=1$ and $y_0$ to be $0$ if $\zeta_0=-1$, the entries $G(b,b)$ can be be re--written as
$$
Q^{d_1} q^{Ld} \prod_{x} q^{x - N_x^1(b) }\prod_y q^{-y - N_y^1(b)},
$$
where $d_1$ is the total number of species $1$ particles and $d$ is the total number of particles. Note that $d$ is unchanged under the dynamics. 
\end{remark}

For $b,\tilde{b} \in \mathcal{B}$, define the relation $\tilde{b} \subseteq b$ to hold if
$$
\{\tilde{x}_1,\ldots,\tilde{x}_{\tilde{d}(1)}\} \subseteq \{x_1,\ldots,x_{d(1)}\}, \quad \{\tilde{y}_1,\ldots,\tilde{y}_{\tilde{d}(-1)}\} \subseteq \{y_1,\ldots,y_{d(-1)}\}, \quad \tilde{z}_0\in \{0,z_0\}.
$$
The relation $\tilde{b}\subseteq b$ holds if and only if $c_{{b}}$ has a nonzero coefficient in $ \Delta^{(L)}(f_{1/2}^d) c_{\tilde{b}}$ for some $d$. 

To complete the proof of Theorem \ref{OPEN}, it suffices to show the next proposition.

\begin{prop}
The functional $Db(\eta,\xi)$ can be written as $q^{ -d(\xi) (d(\xi)-1 ) }G^{-1}SG^{-1}$,
where the constant $d(\xi)$ is defined by $ \vert A_1(\xi) \vert + \vert A_{-1}(\xi)\vert$.
\end{prop}
\begin{proof}
The $G^{-1}$ terms produce products over $A_1(\xi),A_1(\eta),A_{-1}(\xi),A_{-1}(\eta)$, while the $S$ produces products over $A_1(\eta) - A_1(\xi)$ and $A_{-1}(\eta)-A_{-1}(\xi)$. Let $\mathbf{1}$ denote the indicator term. First collect the products of $A_1$ sets, which are
$$
\mathbf{1} \prod_{x \in A_1(\xi)} Q^{-1} q^{-x+\vec{N}_x^1(\xi)} \prod_{x \in A_1(\eta)} Q q^{-x+\vec{N}_x^1(\eta)} \prod_{x \in A_1(\eta) - A_1(\xi)} Q^{-1} q^{x - \vec{N}_x^1(\eta) - 2\vec{N}_x^{-1}(\xi) - 3\cev{N}_x^1(\xi) - 3\cev{N}_x^{-1}(\xi)}.
$$
The $\vec{N}_x^1(\eta) - 2\vec{N}_x^{-1}(\xi) $ comes from the $K_{-1/2}^{-1}$ term on the right, and the $\eta$ in $\vec{N}_x^1(\eta) $ occurs because the $a_x^-$ are applied right to left. The $- 3\cev{N}_x^1(\xi) - 3\cev{N}_x^{-1}(\xi)$ comes from the $k_{1/2}$ on the left. Next, using the cancelation of $q^{x-\vec{N}_x^1(\eta)}$ between the $A_1(\eta)-A_1(\xi)$ and $A_1(\eta)$, we arrive at
$$
\mathbf{1}\prod_{x \in A_1(\xi)} Q^{-1} q^{-x+\vec{N}_x^1(\xi)} \prod_{x \in A_1(\xi)} Q^{-1} q^{-x+\vec{N}_x^1(\eta)} \prod_{x \in A_1(\eta) - A_1(\xi)} q^{- 2\vec{N}_x^{-1}(\xi) - 3\cev{N}_x^1(\xi) - 3\cev{N}_x^{-1}(\xi)}.
$$
Now substitute
$$
\prod_{x \in A_1(\eta) - A_1(\xi)} q^{- 3\cev{N}_x^1(\xi) } = \prod_{x \in A_1(\xi)} q^{-3\vec{N}_x^1(\eta) + 3\vec{N}_x^1(\xi)}
$$
to obtain
$$
\mathbf{1}\prod_{x \in A_1(\xi)} Q^{-2}q^{-2x -2 \vec{N}_x^1(\eta) + 4\vec{N}_x^1(\xi)} \underbrace{\prod_{x \in A_1(\eta) - A_1(\xi)} q^{- 2\vec{N}_x^{-1}(\xi)  - 3\cev{N}_x^{-1}(\xi)}}_{(*)}.
$$
Now, the product of $A_{-1}$ sets yields
$$
\mathbf{1}\prod_{y \in A_{-1}(\xi)} q^{y+ \vec{N}_y^1(\xi)} \prod_{y \in A_{-1}(\eta)} q^{y + \vec{N}_y^1(\eta)} \prod_{y \in A_{-1}(\eta) - A_{-1}(\xi)} q^{-y - \vec{N}_y^1(\eta) - 2\vec{N}_y^{-1}(\xi)},
$$
which again cancels to
$$
\mathbf{1}\prod_{y \in A_{-1}(\xi)} q^{y+ \vec{N}_y^1(\xi)} \underbrace{\prod_{y \in A_{-1}(\xi)} q^{y + \vec{N}_y^1(\eta)} }_{(*)} \prod_{y \in A_{-1}(\eta) - A_{-1}(\xi)} q^{- 2\vec{N}_y^{-1}(\xi)} . 
$$
The product of the two terms underbraced by $(*)$ equals
$$
{\prod_{x \in A_1(\eta) - A_1(\xi)} q^{- 2\vec{N}_x^{-1}(\xi)  - 3\cev{N}_x^{-1}(\xi)}}{\prod_{y \in A_{-1}(\xi)} q^{y + \vec{N}_y^1(\eta)} } = \prod_{x \in A_1(\eta) - A_1(\xi)} q^{- 2\vec{N}_x^{-1}(\xi) } \prod_{y \in A_{-1}(\xi)} q^{y - 2\vec{N}_y^1(\eta) + 3\vec{N}_y^1(\xi)},
$$
where we substituted
$$
{\prod_{x \in A_1(\eta) - A_1(\xi)} q^{ - 3\cev{N}_x^{-1}(\xi)}} =\prod_{y \in A_{-1}(\xi)} q^{-3\vec{N}_y^1(\eta) + 3\vec{N}_y^1(\xi)}.
$$
Collecting and combining all the terms shows that $G^{-1}SG(\eta,\xi)$ equals
$$
\prod_{x \in A_1(\xi)} Q^{-2}q^{-2x -2 \vec{N}_x^1(\eta) + 4\vec{N}_x^1(\xi)}   \prod_{y \in A_{-1}(\eta) - A_{-1}(\xi)} q^{- 2\vec{N}_y^{-1}(\xi)} \prod_{x \in A_1(\eta) - A_1(\xi)} q^{- 2\vec{N}_x^{-1}(\xi) } \prod_{y \in A_{-1}(\xi)} q^{2y - 2\vec{N}_y^1(\eta) + 4\vec{N}_y^1(\xi)}.
$$
Now substitute
$$
\mathbf{1} \prod_{y \in A_{-1}(\eta) - A_{-1}(\xi)} q^{- 2\vec{N}_y^{-1}(\xi)} \prod_{x \in A_1(\eta) - A_1(\xi)} q^{- 2\vec{N}_x^{-1}(\xi) } = \prod_{y \in A_{-1}(\xi)} q^{-2\cev{N}_y^{-1}(\eta) - 2\cev{N}_x^{1}(\eta) + 2\cev{N}_y^{-1}(\xi) + 2 \cev{N}_y^1(\xi)}
$$
to get
$$
\mathbf{1}\prod_{x \in A_1(\xi)} Q^{-2}q^{-2x -2 \vec{N}_x^1(\eta) + 4\vec{N}_x^1(\xi)}   \prod_{y \in A_{-1}(\xi)} q^{2y - 2\vec{N}_y^1(\eta) + 4\vec{N}_y^1(\xi)-2\cev{N}_y^{-1}(\eta) - 2\cev{N}_y^{1}(\eta) + 2\cev{N}_y^{-1}(\xi) + 2 \cev{N}_y^1(\xi)}.
$$
Now substitute
\begin{align*}
\prod_{x \in A_1(\xi)} q^{ 4\vec{N}_x^1(\xi)}   &= q^{2 \vert A_1(\xi)\vert( \vert A_1(\xi)\vert -1)},\\
\prod_{y \in A_{-1}(\xi)} q^{  -2\vec{N}_y^{1}(\eta) - 2\cev{N}_y^{1}(\eta)  } &= q^{-2\vert A_{-1}(\xi)\vert \cdot \vert A_{1}(\eta)\vert},\\
\prod_{y \in A_{-1}(\xi)} q^{  -2\cev{N}_y^{-1}(\eta) } &= q^{ \vert A_{-1}(\xi) \vert ( \vert A_{-1}(\xi) \vert -1)},\\
\prod_{y \in A_{-1}(\xi)} q^{4 \vec{N}_y^1(\xi) + 2\cev{N}_y^1(\xi)} &= q^{2 \vert A_{-1}(\xi)\vert \cdot \vert A_1(\xi)\vert}\prod_{y \in A_{-1}(\xi)} q^{2 \vec{N}_y^1(\xi)} 
\end{align*}
to get that $G^{-1}SG$ equals
\begin{multline*}
\mathbf{1}
 \prod_{x \in A_1(\xi)} (Q^{-2}q^{-2x-2\vec{N}_x^1(\eta)})  
\prod_{y \in A_{-1}(\xi)} q^{2y + 2\cev{N}_y^{-1}(\eta)} \\
q^{2 \vert A_1(\xi)\vert( \vert A_1(\xi)\vert -1)}
q^{-2\vert A_{-1}(\xi)\vert \cdot \vert A_{1}(\eta)\vert}
q^{ \vert A_{-1}(\xi) \vert ( \vert A_{-1}(\xi) \vert -1)}
q^{2 \vert A_{-1}(\xi)\vert \cdot \vert A_1(\xi)\vert}
\prod_{y \in A_{-1}(\xi)} q^{2\vec{N}_y^1(\xi)}.
\end{multline*}
Finally, it remains to show that 
$$
2 \vert A_1(\xi)\vert( \vert A_1(\xi)\vert -1) +  \vert A_{-1}(\xi) \vert ( \vert A_{-1}(\xi) \vert -1) + 2 \vert A_{-1}(\xi)\vert \cdot \vert A_1(\xi)\vert = \vert A_1(\xi)\vert \cdot ( \vert A_1(\xi)\vert -1) + d(\xi)(d(\xi)-1) .
$$
Setting $a= \vert A_1(\xi)\vert$ and $b = \vert A_{-1}(\xi)\vert$, we see that the left--hand--side is
$$
2a(a-1) + b(b-1) + 2ab  = a(a-1) + a^2-a+b^2-b+2ab = a(a-1) + (a+b)^2-(a+b) = a(a-1) + (a+b)(a+b-1),
$$
which equals the right--hand--side. This completes the proof.
\end{proof}

\subsection{Proof of Theorem \ref{braided}}
The main difficulty will be showing that the jump rates are those described by the braided ASEP. Begin by define the stochastic matrix $\check{S}$ by 
\begin{align*}
\check{S} &:= q \cdot G^{-1}( \check{R} - (q-q^{-1})\mathrm{Id})G\\
&= \sum_{i>j} E_{j,i} \otimes E_{i,j} + q ^2\sum_{i>j} E_{i,j} \otimes E_{j,i}+(1-q^2)\sum_{i>j} E_{i,i} \otimes E_{j,j}+  \sum_i  E_{i,i} \otimes E_{i,i}
\end{align*}
Recall that in probabilistic notation, a matrix is stochastic if its rows sum to $1$. Because it suffices to consider the local jump rates, without loss of generality we set $L=2$. The fused permutation $\Sigma_1$ is then the image of the permutation 
$$
 \underbrace{ \sigma_{m}\ldots \sigma_1}_{m \text{ terms}} \cdots \underbrace{\sigma_{2m-2} \cdots  \sigma_{m}\sigma_{m-1}}_{m \text{ terms}}  \underbrace{\sigma_{2m-1} \cdots \sigma_{m+1}\sigma_m}_{m \text{ terms}},
$$
where the number of $\sigma$ terms in total is $m^2$. The matrix entries of $\Sigma_1$ acting on $\mathrm{Sym}^m V \otimes \mathrm{Sym}^m V$ can be described in terms of an auxiliary discrete--time Markov process. The state space of this Markov process consists of particle configurations on the finite lattice $\{1,2,\ldots,2m\}$. The initial configuration has $k_1$ particles located at $\{1,2,\ldots,k_1\}$ and $k_2$ particles located at $\{2m-k_2+1, \ldots, 2m\}$. The update at time $t$ consists of applying the action of $\sigma_{2m-1-t}\cdots \sigma_{m+1-t} \sigma_{m-t}$, where the action of each $\sigma$ is defined by $\check{S}$. The update can be described in the following way. The update occurs in series, starting from the lattice site ${m-1-t}$ and moving to the right to the particle located at $2m-1-t$. The leftmost particle attempts to jump to the right, where the number of spaces it jumps is a truncated geometric random variable with parameter $q^2$. It the particle jumps the maximum distance, so that it lands adjacent to the next particle (or if it started adjacent to the particle on its right), then that particle also attempts to jump to the right, again as a truncated geometric random variable. If at any point, a particle falls short of jumping the maximum number of steps, then all the remaining particles to the right move exactly one step to the left. 

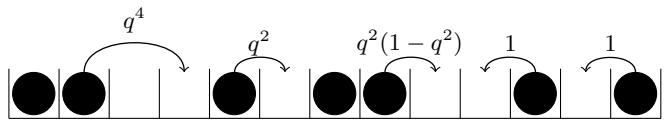
\begin{center}
\begin{tikzpicture}
\usetikzlibrary{arrows}
\draw (0.33,0)--(9,0);

\foreach \x in {0,1,2,3,4,5,6,7,8,9,10,11,12,13}
	\draw (0.33+\x/1.5,0)--(0.33+\x/1.5,0.66);

\foreach \x in {0,1,4,6,7,10,12}
	\draw[fill=black] (0.66+\x/1.5,0.33) circle (8pt);
	
\node at (2,1.3) {$q^4$};	
\node at (1.33,0.5) (abc) {};
\node at (2.66,0.5) (abd) {};
\draw (abc) edge[out=90,in=90,->, line width=0.5pt] (abd);

\node at (3.66,1) {$q^2$};
\node at (3.33,0.5) (abc) {};
\node at (4,0.5) (abd) {};
\draw (abc) edge[out=90,in=90,->, line width=0.5pt] (abd);

\node at (5.66,1) {$q^2(1-q^2)$};
\node at (5.33,0.5) (abc) {};
\node at (6,0.5) (abd) {};
\draw (abc) edge[out=90,in=90,->, line width=0.5pt] (abd);

\node at (7,1) {$1$};
\node at (7.33,0.5) (abc) {};
\node at (6.66,0.5) (abd) {};
\draw (abc) edge[out=90,in=90,->, line width=0.5pt] (abd);

\node at (8.33,1) {$1$};
\node at (8.66,0.5) (abc) {};
\node at (8,0.5) (abd) {};
\draw (abc) edge[out=90,in=90,->, line width=0.5pt] (abd);

\end{tikzpicture}
\end{center}
After $m$ updates, let $J_{m,k_1,k_2}$ denote the number of particles located  at $\{m+1, \ldots,2m\}$.

\begin{lemma}
The $(k_1,k_2;l_1,l_2)$ matrix entry of the action of $\Sigma_1$ equals $1_{k_1+k_2=l_1+l_2}\mathbb{P}(J_{m,k_1,k_2}=l_2)$. 
\end{lemma}
\begin{proof}
As explained in \cite{crampe2020fused} and \cite{crampebax}, the action of each $\Sigma_i$ corresponds to an $R$--matrix. The $R$--matrix satisfies a fusion equation (see (20) of \cite{crampebax})
$$
\Sigma_1 =   \underbrace{ \sigma_{m}\ldots \sigma_1}_{m \text{ terms}} \cdots \underbrace{\sigma_{2m-2} \cdots  \sigma_{m}\sigma_{m-1}}_{m \text{ terms}}  \underbrace{\sigma_{2m-1} \cdots \sigma_{m+1}\sigma_m}_{m \text{ terms}} P,
$$
where $P$ is the symmetrizer from ${\mathbb{C}^2}^{\otimes 2m}$ to $\mathrm{Sym}^m \mathbb{C}^2 \otimes \mathrm{Sym}^m \mathbb{C}^2$. The symmetrizer is simply the map $\Phi$ which fuses adjacent sites; see the accompanying image. See also section 3.2 of \cite{KuanCMP}, or \cite{KuanSF}, where the $P$ is described as $\Phi$. The fusion equation allows the ``fission'' map $\Lambda$ to be arbitrary (note that there is no $P$ on the left), so for convenience we choose the $k_1$ particles to be farthest to the left and the $k_2$ particles to be farthest to the right. This gives the initial condition of the auxiliary process.The $m^2$ $\sigma$'s describe the update of the auxiliary process after $m$ time steps. The preimage of $(l_1,l_2)$ under $\Phi$ is simply all particle configurations where there are $l_2$ particles in the right half of the lattice. The probability of having exactly $l_2$ particles in the right half is just $\mathbb{P}(J_{m,k_1,k_2}=l_2)$. This completes the proof.

\begin{figure}\label{LP}
\begin{center}
\begin{tikzpicture}
\draw (2,2) circle (3pt);
\draw (4,2) circle (3pt);
\fill[black] (5,2) circle (3pt);
\fill[black] (1,2) circle (3pt);
\draw (3,2) circle (3pt);
\fill[black] (6,2) circle (3pt);
\draw [very thick](0.5,2) -- (6.5,2);
\draw (8,2) node (a) { $ $};
\draw (8,0) node (b) { $ $};
\draw (0,2) node (c) { $$};
\draw (0,0) node (d) { $$};
\draw (a) edge[out=-45,in=45,->, line width=0.5pt] (b);
\draw (d) edge[out=135,in=-135,->, line width=0.5pt] (c);
\draw (0.7,2) node {\Huge $[$};
\draw (3.3,2) node {\Huge $]$};
\draw (3.7,2) node {\Huge $[$};
\draw (6.3,2) node {\Huge $]$};
\draw (0.7,0) node {\Huge $[$};
\draw (3.3,0) node {\Huge $]$};
\draw (3.7,0) node {\Huge $[$};
\draw (6.3,0) node {\Huge $]$};
\draw [very thick](0.5,-0) -- (6.5,-0);
\draw (2,0.6) circle (3pt);
\draw (2,0.2) circle (3pt);
\fill[black] (2,-0.2) circle (3pt);

\draw (5,0.6) circle (3pt);
\fill[black] (5,0.2) circle (3pt);
\fill[black] (5,-0.2) circle (3pt);

\draw (9,1) node {$\Phi$};
\draw (-1,1) node {$\Lambda$};
\end{tikzpicture}
\end{center}
\end{figure}

\end{proof}

We return to a discussion of the auxiliary Markov process. Note that the first $m-k_1$ time updates are deterministic, with each update moving the block of $k_2$ particles one step to the left. Thus, it is equivalent to consider an initial condition with particles located at $\{1,2,\ldots,k_1\}$ and $\{m-k_2+k_1+1,\ldots,m+k_1\}$ with only $k_1$ time updates. We will proceed by induction on $k_1$. When $k_1=0$, then clearly $\mathbb{P}(J_{m,0,k_2}=0=1)$, which matches the jump rate. For the inductive step, note that the jump rates satisfy a recurrence relation:
\begin{lemma}\label{recur}
Let
$$
p_m(k_1,k_2;l_1,l_2):=\binom{k_1}{l_2}_q (q^{2(m-k_2)};q^{-2})_{k_1-l_2} q^{2(m-k_2-k_1+l_2)l_2}.
$$
Then
$$
p_m(k_1,k_2;l_1,l_2) := q^{2(m-k_2-k_1+l_2)} p_m(k_1-1,k_2;l_1,l_2-1) + (1-q^{2(m-k_2-k_1+l_2+1)}) p_m(k_1-1,k_2;l_1,l_2)
$$
\end{lemma}
\begin{proof}

Plugging in the expression for $p_m$, the right--hand--side is
\begin{multline*}
q^{2(m-k_2-k_1+l_2)}  \left\{ \begin{array}{c} k_1-1 \\ l_2-1 \end{array} \right\}_{q^2} (q^{2(m-k_2)};q^{-2})_{k_1-l_2} q^{2(m-k_2-k_1+l_2)(l_2-1)}  \\
+ (1-q^{2(m-k_2-k_1+l_2+1)}) \left\{ \begin{array}{c} k_1-1 \\ l_2 \end{array} \right\}_{q^2} (q^{2(m-k_2)};q^{-2})_{k_1-1-l_2}q^{2(m-k_2-k_1+l_2+1)l_2} .
\end{multline*}
First, note that
$$
(1-q^{2(m-k_2-k_1+l_2+1)}) (q^{2(m-k_2)};q^{-2})_{k_1-1-l_2} = (q^{2(m-k_2)};q^{-2})_{k_1-l_2},
$$
so the right--hand--side factors as 
$$
\left(  \left\{ \begin{array}{c} k_1-1 \\ l_2-1 \end{array} \right\}_{q^2} + q^{2l_2}\left\{ \begin{array}{c} k_1-1 \\ l_2 \end{array} \right\}_{q^2}   \right)(q^{2(m-k_2)};q^{-2})_{k_1-l_2}, q^{2(m-k_2-k_1+l_2)l_2} .
$$
Using the identity 
$$
 \left\{ \begin{array}{c} k_1-1 \\ l_2-1 \end{array} \right\}_{q^2} + q^{2l_2}\left\{ \begin{array}{c} k_1-1 \\ l_2 \end{array} \right\}_{q^2}  = \left\{ \begin{array}{c} k_1 \\ l_2 \end{array} \right\}_{q^2} 
$$
completes the proof of the lemma.
\end{proof}
We will show that the quantity $\mathbb{P}(J_{m,k_1,k_2}=l_2)$ satisfies the same recurrence relation. First, introduce the notation to let $y_1(t)<\ldots< y_{k_2}(t)$ denote the location of the block of $k_2$ particles at time $t$. So $y_1(0)=m-k_2+k_1+1$ up to $y_{k_2}(0)=m+k_1$. Note that since $y_{j+1}(t) \geq y_j(t)+1$ for all times $t$ and all $j\in \{1,\ldots,k_2\}$, thus we can conclude $J_{m,k_1,k_2} \geq y_1(k_1)-(m+1)+k_2$. In fact, the $y_j$ particles can never jump to the right, because the update terminates at lattice site $m_k+k_1-t$; so the reverse inequality $J_{m,k_1,k_2} \leq y_1(k_1)-(m+1)+k_2$ holds as well. Therefore, $J_{m,k_1,k_2}=y_1(k_1)-(m+1)+k_2$, so analyzing $J_{m_1,k_1,k_2}$ is equivalent to analyzing $y_1(k_1)$.

After each time update, the particle $y_1$ either jumps one step to the left or does not jump, so analyzing $J_{m,k_1,k_2}$ is equivalent to analyzing the times at which $y_1$ does not jump. For each $t \in \{1,2,\ldots,k_1\}$ let $\mathcal{S}(m,t,k_2) \subseteq \{1,\ldots,t\}$ denote the set of times at which $y_1$ does not jump. Note that this set only depends on the relative positions of the particles with each other, which does not depend on $k_1$, and therefore it is notationally sound to write $\mathcal{S}(m,t,k_2)$ with no dependence on $k_1$. We thus have that $J_{m,k_1,k_2}= \vert \mathcal{S}(m,k_1,k_2)\vert$.

To complete the inductive step, it suffices to show that
\begin{multline*}
\mathbb{P}(\vert \mathcal{S}(m,k_1,k_2)\vert = l_2) = q^{2(m-k_2-k_1+l_2)} \mathbb{P}(\vert \mathcal{S}(m,k_1-1,k_2)\vert = l_2-1)  \\
+ (1-q^{2(m-k_2-k_1+l_2+1)} ) \mathbb{P}(\vert \mathcal{S}(m,k_1-1,k_2)\vert = l_2) .
\end{multline*}
To do this, it suffices to show the conditional probabilities
\begin{align*}
\mathbb{P}(\vert \mathcal{S}(m,k_1,k_2)\vert = l_2 \big| \vert \mathcal{S}(m,k_1-1,k_2)\vert = l_2-1) &=  q^{2(m-k_2-k_1+l_2)} \\
\mathbb{P}(\vert \mathcal{S}(m,k_1,k_2)\vert = l_2 \big| \vert \mathcal{S}(m,k_1-1,k_2)\vert = l_2) &= 1-q^{2(m-k_2-k_1+l_2+1)} ,
\end{align*}
since $\{ \vert \mathcal{S}(m,k_1,k_2)\vert = l_2 \} \subseteq \{\vert \mathcal{S}(m,k_1-1,k_2)\vert = l_2-1)\} \cup \{\vert \mathcal{S}(m,k_1-1,k_2)\vert = l_2)\}.$ Note that once the first equality is shown, the second follows from 
\begin{align*}
&\mathbb{P}(\vert \mathcal{S}(m,k_1,k_2)\vert = l_2 \big| \vert \mathcal{S}(m,k_1-1,k_2)\vert = l_2) \\
&= 1- \mathbb{P}(\vert \mathcal{S}(m,k_1,k_2)\vert = l_2+1 \big| \vert \mathcal{S}(m,k_1-1,k_2)\vert = (l_2+1)-1)\\
&=  1-q^{2(m-k_2-k_1+l_2+1)}.
\end{align*}
To show this first inequality, note that after $k_1-1$ time updates, if the the set $\mathcal{S}$ has magnitude $l_2-1$, this means the $y_1$ particle has jumped $k_1-l_2$ times, so is located at $m-k_2+l_2+1$. In order for $y_1$ to not jump, we require every particle to jump as far to the right as possible. The update begins at lattice site $1$, where there is a particle, and ends at the lattice site $m-k_2+l_2$. There are $k_1-1$ particles located in between the two lattice sites, so the total distance jumped must be $(m-k_2+l_2-1)-(k_1-1)=m-k_2+l_2-k_1$, which has a probability $q^{2(m-k_2+l_2-k_1)}$. This completes the proof.

\section{Appendix}
\subsection{Examples}
Fix $q \in [0,1]$. Let $S$ denote the $4\times 4$ stochastic matrix acting on $\mathbb{R}^2 \otimes \mathbb{R}^2$ by 
$$
\left(
\begin{array}{cccc}
 1 & 0 & 0 & 0 \\
 0 & 1-q^2 & q^2 & 0 \\
 0 & 1 & 0 & 0 \\
 0 & 0 & 0 & 1 \\
\end{array}
\right).
$$
The product $S_{23}S_{12}S_{34}S_{23}-\mathrm{Id}_{16}$ is the $16\times 16$ ``fused'' matrix acting on $\mathbb{R}^2 \otimes \mathbb{R}^2 \otimes \mathbb{R}^2 \otimes \mathbb{R}^2 $ by 
$$
\tiny
\left(
\begin{array}{cccccccccccccccc}
 0 & 0 & 0 & 0 & 0 & 0 & 0 & 0 & 0 & 0 & 0 & 0 & 0 & 0 & 0 & 0 \\
 0 & -q^2 & q^2-q^4 & 0 & q^4 & 0 & 0 & 0 & 0 & 0 & 0 & 0 & 0 & 0 & 0 & 0 \\
 0 & 1-q^2 & -q^4+q^2-1 & 0 & 0 & 0 & 0 & 0 & q^4 & 0 & 0 & 0 & 0 & 0 & 0 & 0 \\
 0 & 0 & 0 & q^2 \left(q^4-q^2-1\right) & 0 & q^2-q^4 & q^4-q^6 & 0 & 0 & q^4-q^6 &
   q^6-q^8 & 0 & q^8 & 0 & 0 & 0 \\
 0 & 1 & 0 & 0 & -1 & 0 & 0 & 0 & 0 & 0 & 0 & 0 & 0 & 0 & 0 & 0 \\
 0 & 0 & 0 & 1-q^2 & 0 & q^2-1 & 0 & 0 & 0 & 0 & 0 & 0 & 0 & 0 & 0 & 0 \\
 0 & 0 & 0 & 1-q^2 & 0 & 0 & -1 & 0 & 0 & q^2 & 0 & 0 & 0 & 0 & 0 & 0 \\
 0 & 0 & 0 & 0 & 0 & 0 & 0 & -q^2 & 0 & 0 & 0 & q^2-q^4 & 0 & q^4 & 0 & 0 \\
 0 & 0 & 1 & 0 & 0 & 0 & 0 & 0 & -1 & 0 & 0 & 0 & 0 & 0 & 0 & 0 \\
 0 & 0 & 0 & 1-q^2 & 0 & 0 & q^2 & 0 & 0 & -1 & 0 & 0 & 0 & 0 & 0 & 0 \\
 0 & 0 & 0 & 1-q^2 & 0 & 0 & 0 & 0 & 0 & 0 & q^2-1 & 0 & 0 & 0 & 0 & 0 \\
 0 & 0 & 0 & 0 & 0 & 0 & 0 & 1-q^2 & 0 & 0 & 0 & -q^4+q^2-1 & 0 & 0 & q^4 & 0 \\
 0 & 0 & 0 & 1 & 0 & 0 & 0 & 0 & 0 & 0 & 0 & 0 & -1 & 0 & 0 & 0 \\
 0 & 0 & 0 & 0 & 0 & 0 & 0 & 1 & 0 & 0 & 0 & 0 & 0 & -1 & 0 & 0 \\
 0 & 0 & 0 & 0 & 0 & 0 & 0 & 0 & 0 & 0 & 0 & 1 & 0 & 0 & -1 & 0 \\
 0 & 0 & 0 & 0 & 0 & 0 & 0 & 0 & 0 & 0 & 0 & 0 & 0 & 0 & 0 & 0 \\
\end{array}
\right).
$$
This matrix, denoted by $L$, is the generator of a continuous time Markov chain. 
For some other parameter $s\in [0,1]$, define the Markov operator $\Lambda$ from $\mathrm{Sym}^2\mathbb{R}^2 \otimes \mathrm{Sym}^2\mathbb{R}^2 $ to $\mathbb{R}^2 \otimes \mathbb{R}^2 \otimes \mathbb{R}^2 \otimes \mathbb{R}^2 $ by 
$$
\left(
\begin{array}{cccccccccccccccc}
 1 & 0 & 0 & 0 & 0 & 0 & 0 & 0 & 0 & 0 & 0 & 0 & 0 & 0 & 0 & 0 \\
 0 & \frac{1}{s+1} & \frac{s}{s+1} & 0 & 0 & 0 & 0 & 0 & 0 & 0 & 0 & 0 & 0 & 0 & 0 &
   0 \\
 0 & 0 & 0 & 1 & 0 & 0 & 0 & 0 & 0 & 0 & 0 & 0 & 0 & 0 & 0 & 0 \\
 0 & 0 & 0 & 0 & \frac{1}{s+1} & 0 & 0 & 0 & \frac{s}{s+1} & 0 & 0 & 0 & 0 & 0 & 0 &
   0 \\
 0 & 0 & 0 & 0 & 0 & \frac{1}{(s+1)^2} & \frac{s}{(s+1)^2} & 0 & 0 &
   \frac{s}{(s+1)^2} & \frac{s^2}{(s+1)^2} & 0 & 0 & 0 & 0 & 0 \\
 0 & 0 & 0 & 0 & 0 & 0 & 0 & \frac{1}{s+1} & 0 & 0 & 0 & \frac{s}{s+1} & 0 & 0 & 0 &
   0 \\
 0 & 0 & 0 & 0 & 0 & 0 & 0 & 0 & 0 & 0 & 0 & 0 & 1 & 0 & 0 & 0 \\
 0 & 0 & 0 & 0 & 0 & 0 & 0 & 0 & 0 & 0 & 0 & 0 & 0 & \frac{1}{s+1} & \frac{s}{s+1} &
   0 \\
 0 & 0 & 0 & 0 & 0 & 0 & 0 & 0 & 0 & 0 & 0 & 0 & 0 & 0 & 0 & 1 \\
\end{array}
\right)
$$
and the deterministic projection $\Phi$ from $\mathbb{R}^2 \otimes \mathbb{R}^2 \otimes \mathbb{R}^2 \otimes \mathbb{R}^2 $ to $\mathrm{Sym}^2\mathbb{R}^2 \otimes \mathrm{Sym}^2\mathbb{R}^2 $ by 
$$
\left(
\begin{array}{ccccccccc}
 1 & 0 & 0 & 0 & 0 & 0 & 0 & 0 & 0 \\
 0 & 1 & 0 & 0 & 0 & 0 & 0 & 0 & 0 \\
 0 & 1 & 0 & 0 & 0 & 0 & 0 & 0 & 0 \\
 0 & 0 & 1 & 0 & 0 & 0 & 0 & 0 & 0 \\
 0 & 0 & 0 & 1 & 0 & 0 & 0 & 0 & 0 \\
 0 & 0 & 0 & 0 & 1 & 0 & 0 & 0 & 0 \\
 0 & 0 & 0 & 0 & 1 & 0 & 0 & 0 & 0 \\
 0 & 0 & 0 & 0 & 0 & 1 & 0 & 0 & 0 \\
 0 & 0 & 0 & 1 & 0 & 0 & 0 & 0 & 0 \\
 0 & 0 & 0 & 0 & 1 & 0 & 0 & 0 & 0 \\
 0 & 0 & 0 & 0 & 1 & 0 & 0 & 0 & 0 \\
 0 & 0 & 0 & 0 & 0 & 1 & 0 & 0 & 0 \\
 0 & 0 & 0 & 0 & 0 & 0 & 1 & 0 & 0 \\
 0 & 0 & 0 & 0 & 0 & 0 & 0 & 1 & 0 \\
 0 & 0 & 0 & 0 & 0 & 0 & 0 & 1 & 0 \\
 0 & 0 & 0 & 0 & 0 & 0 & 0 & 0 & 1 \\
\end{array}
\right).
$$
One can compute that $\Lambda L \Phi$ acts on $\mathrm{Sym}^2 \mathbb{R}^2 \otimes \mathrm{Sym}^2 \mathbb{R}^2$ as 
$$
\left(
\begin{array}{ccccccccc}
 0 & 0 & 0 & 0 & 0 & 0 & 0 & 0 & 0 \\
 0 & -q^4 & 0 & q^4 & 0 & 0 & 0 & 0 & 0 \\
 0 & 0 & q^2 \left(q^4-q^2-1\right) & 0 & \left(1-q^2\right) \left(q^3+q\right)^2 &
   0 & q^8 & 0 & 0 \\
 0 & 1 & 0 & -1 & 0 & 0 & 0 & 0 & 0 \\
 0 & 0 & 1-q^2 & 0 & q^2-1 & 0 & 0 & 0 & 0 \\
 0 & 0 & 0 & 0 & 0 & -q^4 & 0 & q^4 & 0 \\
 0 & 0 & 1 & 0 & 0 & 0 & -1 & 0 & 0 \\
 0 & 0 & 0 & 0 & 0 & 1 & 0 & -1 & 0 \\
 0 & 0 & 0 & 0 & 0 & 0 & 0 & 0 & 0 \\
\end{array}
\right),
$$
which has no dependence on $s$. Furthermore, if $G$ is defined to be the diagonal matrix with entries given by $(1,1,1,q^{-2},q^{-1}(1+q^2)^{-1},1,q^{-4},q^{-2},1)$, then 
$$
(\Lambda L \Phi)^* = G^{-2} \cdot \Lambda L \Phi \cdot G^2,
$$
where the $ ^*$ denotes transpose. This means that $G^{-2}$ gives the reversible measure for the continuous--time Markov process generated by $\Lambda L \Phi$. These are the same reversible measures as in Proposition 2.6 of \cite{KIMRN}.

\bibliographystyle{alpha}
\bibliography{SchurWeylNotes}

\end{document}